\documentclass{amsart}
\usepackage[pagewise]{lineno}

\newcommand{\Q}{\mathbb{Q}}

\newcommand{\R}{\mathbb{R}}
\newcommand{\N}{\mathbb{N}}

\newcommand{\ran}{\operatorname{ran}}
\newcommand{\avg}{\operatorname{avg}}
\newcommand{\intdiff}{\ooalign{\hidewidth\raise1ex\hbox{.}\hidewidth\cr$-$\cr}}

\newcommand{\KO}{\mathcal{O}}
\newcommand{\CK}{\operatorname{CK}}
\newcommand{\true}{\top}
\newcommand{\false}{\bot}
\DeclareMathOperator*{\infconj}{\bigwedge\mkern-15mu\bigwedge}
\DeclareMathOperator*{\infdis}{\bigvee\mkern-15mu\bigvee}

\newcommand{\M}{\mathcal{M}}

\theoremstyle{theorem}
\newtheorem{theorem}{Theorem}[section]
\newtheorem{lemma}[theorem]{Lemma}

\newtheorem{proposition}[theorem]{Proposition}

\theoremstyle{definition}
\newtheorem{definition}[theorem]{Definition}

\begin{document}

\title{On the complexity of the theory of a computably presented metric structure}

\author{Caleb Camrud}
\address{Department of Mathematics\\
Iowa State University\\
Ames, Iowa 50011}
\email{ccamrud@iastate.edu}

\author{Isaac Goldbring}
\thanks{Goldbring was partially supported by NSF grant DMS-2054477.}
\address{Department of Mathematics\\
University of California, Irvine\\
340 Rowland Hall (Bldg.\# 400)\\
Irvine, CA 92697-3875}
\email{isaac@math.uci.edu}

\author{Timothy H. McNicholl}
\address{Department of Mathematics\\
Iowa State University\\
Ames, Iowa 50011}
\email{mcnichol@iastate.edu}

\begin{abstract}
We consider the complexity (in terms of the arithmetical hierarchy) of the various 
quantifier levels of the diagram of a computably presented metric structure.  
As the truth value of a sentence of continuous logic may be any real in $[0,1]$, we introduce two kinds 
of diagrams at each level: the \emph{closed} diagram, which encapsulates weak inequalities of the form $\phi^\M \leq r$, and the \emph{open} diagram, which encapsulates strict inequalities of the form $\phi^\M < r$.  
We show that the closed and open $\Sigma_N$ diagrams are $\Pi^0_{N+1}$ and $\Sigma_N$ respectively, and that the 
closed and open $\Pi_N$ diagrams are $\Pi^0_N$ and $\Sigma^0_{N + 1}$ respectively.  
We then introduce effective infinitary formulas of continuous logic and extend our results to the hyperarithmetical hierarchy.  Finally, we demonstrate that our results are optimal.
\end{abstract}
\maketitle

\section{Introduction}\label{sec:intro}

Suppose $\mathcal{A}$ is a \emph{computably presented} countable structure, that is, we have numbered the elements of its domain so that the resulting operations and relations on the natural numbers are computable.
A longstanding and ongoing line of inquiry in computable model theory is to study the complexity of the elementary (i.e. complete) diagram of such models at the various quantifier levels.  In particular, such a model is said to be \emph{$N$-decidable} if the set of the $\Sigma_N$-sentences of its elementary diagram is computable.  A seminal result in this direction is the theorem of Moses and Chisholm that there is a computable linear order that is $n$-decidable for all $n$ yet not decidable \cite{Chisholm.Moses.1998}.  More recently, Fokina et. al. have investigated index sets of $n$-decidable models; i.e. the complexity of classifying such models \cite{Fokina.Goncharov.Kharizanova.Kudinov.Turetski.2015}.  More results along these lines can be found in the survey by Fokina, Harizanov, and Melnikov \cite{Fokina.Harizanov.Melnikov.2014}.

Here, we wish to initiate a similar program for metric structures in the context of continuous logic as expounded in 
\cite{Ben-Yaacov.Berenstein.Henson.Usvyatsov.2008}.  We use the framework for studying the computability of metric structures that has evolved over approximately the past decade (see e.g. \cite{Melnikov.2013}, \cite{Franklin.McNicholl.2020} ).  There are two difficulties that must be confronted at the outset.
One difficulty is that for a sentence $\phi$ of continuous logic, the truth value of $\phi$ can be any real in $[0,1]$, with $0$ representing truth and $1$ representing falsity.  Another difficulty is that the domain of a typical metric structure is uncountable, whence the inclusion of parameters in our sentences would immediately pose complications for a computability-theoretic analysis.  
Our solution to the first difficulty is to study two kinds of diagrams: \emph{closed} diagrams, corresponding to inequalities of the form $\phi^\M \leq r$, and \emph{open} diagrams, corresponding to inequalities of the form $\phi^\M < r$.  (Here $\phi^\M$ is the truth-value of $\phi$ in the model $\M$.)
We leave consideration of possible solutions of the second obstacle for future work.  Consequently, we only consider parameter-free sentences.

In the classical case, the complexity of the levels of a diagram of a computably presented model is very straightforward: the collection of true $\Sigma_N$ sentences is $\Sigma^0_N$ and the collection of true $\Pi_N$ sentences is $\Pi^0_N$.  
True arithmetic demonstrates that these bound are optimal.  We find, however, that in the context of continuous logic, the relation is not so straightforward.  For example, in our first main result (Theorem \ref{thm:fin.up}), we show that the 
closed $\Sigma_N$ diagram is $\Pi^0_{N+1}$, so that we obtain neither the expected quantifier nor the expected level of complexity.  This result may seem surprising at first due to its dissonance with the classical case.  However, some
reflection on the nature of computation with real numbers will likely reveal it is the only answer possible.  Nevertheless, in our second main result (Theorem \ref{thm:fin.lwr}), we show that our upper bounds in the finite case are indeed optimal.  

We then extend our results to infinitary continuous logic.  In this context, we use the hyperarithmetical hierarchy to gauge complexity.  The theory of infinitary continuous logic has been previously studied in  \cite{Ben-Yaacov.Iovino.2009} and \cite{Eagle.2017}.   
To the best of our knowledge, this is the first paper to consider effective infinitary logic for metric structures.  
As might be expected, our results for infinitary logic (Theorems \ref{thm:inf.up} and \ref{thm:inf.lwr}), parallel our findings for finitary logic.  However, the availability of infinite disjunctions yields simpler demonstrations of the lower bounds.  

The paper is organized as follows.  Section \ref{sec:bak} covers relevant background from computability theory, computable analysis, and continuous logic. Section \ref{sec:prelim} lays out the framework for effective infinitary continuous logic as well as some combinatorial results which support our work on finitary logic.  Upper and lower bounds for the finitary case as presented in Sections \ref{sec:fin.up} and \ref{sec:fin.lwr} respectively.  The upper and lower bounds for the infinitary case are demonstrated in Section \ref{sec:inf}.
Finally, Section \ref{sec:conclusion} summarizes our findings and presents some avenues for further investigation.

\section{Background}\label{sec:bak}

\subsection{Background from continuous Logic}

We generally follow the framework of  \cite{Ben-Yaacov.Berenstein.Henson.Usvyatsov.2008}.  
However, we limit our connectives to $\neg$, $\frac{1}{2}$, and $\intdiff$.  
The universal and existential quantifiers are replaced by `$\sup$' and `$\inf$' respectively.
In the following, by \emph{language}, we mean a signature for a metric structure.  A language in this sense includes a modulus of uniform continuity for each predicate symbol and each function symbol.   When $\M$ is an $L$-structure, we denote the domain of $\M$ as $|\M|$.

The $\Sigma_N$ and $\Pi_N$ wff's of a language $L$ are defined as in the classical case.  For example, if $\phi$ is a quantifier-free wff of $L$, then $\inf_{x_1} \sup_{x_2} \phi$ is a $\Sigma_2$ wff of $L$.

The language $L_{\omega_1\omega}$ is considered in the sense of Eagle in \cite{Eagle.2017} as opposed to the language given by Ben-Yaacov and Iovino in \cite{Ben-Yaacov.Iovino.2009}. The key distinction is that $L_{\omega_1\omega}$ in \cite{Eagle.2017} does not require every infinitary formula to have a modulus of uniform convergence, while the language of \cite{Ben-Yaacov.Iovino.2009} does.   Adding this extra condition complicates the effective encoding of the computable infinitary formulas.  However, as we shall see later, our results will hold 
in any reasonable effectivization of the framework of Ben-Yaacov and Iovino.

A key terminological difference with classical infinitary logic is that $\infdis$ is used for infinite conjunction and $\infconj$ for infinite disjunction.  That is, $\infdis_n$ is interpreted as $\sup_n$ and $\infconj_n$ is interpreted as $\inf_n$.  The reasons for this are clear when considering the ordered set of real numbers as a lattice.

\subsection{Background from computability theory}

Familiarity with standard computability-theoretic concepts like computable enumerability, oracle computability, the arithmetical hierarchy, and the relationship between each of these is assumed. A thorough treatment of these subjects can be found in \cite{Soare.1987}, \cite{Cooper.2004}.  For background on the hyperarithmetical hierarchy, see \cite{Ash.Knight.2000} and \cite{Sacks.1990}.

Let $\mathcal{O}$ denote Kleene's system of notations for the computable ordinals.  
If $\alpha < \omega_1^{\CK}$, then $\langle \alpha \rangle$ denotes the set of all notations for $\alpha$.
	
A real number $r$ is \emph{computable} if there is an effective procedure which, given $k\in \N$, produces a rational number $q$ such that $|r-q|<2^{-k}$.  A sequence $(r_n)_{n \in \N}$ of reals is computable if it is computable uniformly in $n$.  By an \emph{index} of such a sequence we mean an index of a Turing machine that computes it.

Suppose $(M,d)$ and $(M',d')$ are metric spaces, and let $\Gamma : M \rightarrow M'$.  
A map $\Delta : \N \rightarrow \N$ is called a \emph{modulus of continuity} for $\Gamma$ if 
$d(a,b)\leq 2^{-\Delta(k)}$ whenever $d'(\Gamma(a),\Gamma(b))\leq 2^{-k}$.
A map $\Gamma:M \to M'$ is called \emph{effectively uniformly continuous} if it has a computable modulus of uniform continuity.  

In the following, $L$ denotes an effectively numbered language with uniformly computable moduli of uniform continuity.   That is, there is an algorithm that given a number assigned to a predicate or function symbol $\phi$ computes the modulus function of $\phi$.  Moreover, unless otherwise mentioned, every structure will be assumed to be an $L$-structure.

Our framework for the computability of metric structures is essentially that in \cite{Franklin.McNicholl.2020}.  Given a structure $\M$ and $A\subseteq |\M|$, we define the \emph{algebra generated by} $A$ to be the smallest subset of $|\M|$ containing $A$ that is closed under every function of $\M$.   A pair $(\mathcal{M},g)$ is called a \emph{presentation} of $\M$ if $g:\N\to |\M|$ is a map such that the algebra generated by $\ran(g)$ is dense. We use $\M^\sharp$ to denote presentations of a structure $\M$. Given a presentation $\M^\sharp=(\M,g)$, every $a\in\ran(g)$ is called a \emph{distinguished point} of $\M^\sharp$, and each point in the algebra generated by the distinguished points is called a  \emph{rational point} of $\M^\sharp$. The set of all rational points of $\M^\sharp$ is denoted $\mathbb{Q}(\M^\sharp)$.
By an \emph{open rational ball} of $\M^\sharp$ we mean an open ball of $\M$ whose radius is rational and whose center is a rational point of $\M^\sharp$.  By a \emph{rational cover} of $\M^\sharp$ we mean a finite set of rational balls of $\M^\sharp$ that covers $|\M|$.

A presentation $\M^\sharp$ is \emph{computable} if the predicates of $\M$ are uniformly computable on the rational points of $\M^\sharp$. Since the metric is a binary predicate on $\M$, this entails that the distance between any two rational points is uniformly computable. We say that a metric structure is \emph{computably presentable} if it has a computable presentation. We say that a presentation $\M^\sharp$ is \emph{computably compact} if the set of its rational covers is computably enumerable.  Lastly, we define an \emph{index} of a computable presentation $\M^\sharp$ to be a code of a Turing machine that computes the predicates of $\M$ on the rational points of $\M^\sharp$.

\section{Preliminaries}\label{sec:prelim}

\subsection{Preliminaries from classical logic and computability}

We begin with some relational notation which will facilitate the statements of many of our results and their proofs.

\begin{definition}\label{def:e.u.R}
Let $N \in \N$, and suppose $R \subseteq \N^{N + 1}$. 
\begin{enumerate}
	\item $\neg R = \N^{N + 1} - R$.\label{def:e.u.R::neg}

	\item $\vec{\exists} R = \{n \in \N\ :\ \exists x_1 \forall x_2 \ldots Qx_N\ R(n, x_1, \ldots, x_N)\}$.\label{def:e.u.R::exists}
	
	\item $\vec{\forall} R = \{n \in \N\ :\ \forall x_1 \exists x_2 \ldots Qx_N\ R(n, x_1, \ldots, x_N)\}$.\label{def:e.u.R::forall}
\end{enumerate} 
\end{definition}

In Definition \ref{def:e.u.R}.\ref{def:e.u.R::exists}, $Q$ denotes the quantifier $\forall$ if $N$ is even and $\exists$ if $N$ is odd.  Similarly, in Definition \ref{def:e.u.R}.\ref{def:e.u.R::forall}, 
$Q$ denotes the quantifier $\forall$ if $N$ is odd and $\exists$ if $N$ is even. 
 We will follow these conventions in the sequel.  

Given $R\subseteq \N^{N+1}$, we also set 	
\[
R^* = \{(n, x_1, \ldots, x_N)\ \in \N^{N + 1}\ :\ \forall x_1' \leq x_1 \exists x_2' \leq x_2 \ldots Qx_N' \leq x_N\ R(n, x_1', \ldots, x_N')\}.
\]
Note that $R \equiv_{\rm T} R^*$.  Finally, let $\chi_R$ denote the characteristic (indicator) function of $R$.  

We fix a uniformly computable family $(R_N)_{N \in \N}$ of relations so that for each $N \in \N$, 
$R_{2N} \cup R_{2N + 1} \subseteq \N^{N + 2}$, $\vec{\forall} R_{2N}$ is $\Pi^0_{N + 1}$-complete, and $\vec{\exists} R_{2N + 1}$ is $\Sigma^0_{N+1}$-complete.

\subsection{Preliminaries from continuous logic}

We begin by formally defining the open and closed diagrams of a metric structure.

\begin{definition}\label{def:diagrams}
Let $\M$ be an $L$-structure.  In the following, $\phi$ ranges over sentences of $L$ and $q$ ranges over $[0,1]\cap\Q$.
\begin{enumerate}
	\item The \emph{closed (resp. open) quantifier-free diagram of $\M$} is the set of all pairs $(\phi, q)$ so that 
	$\phi$ is quantifier-free and $\phi^\M \leq q$ (resp. $\phi^\M < q$). 
		
	\item For every positive integer $N$, the \emph{closed (resp. open) $\Pi_N$ diagram of $\M$} is the set of all pairs 
	$(\phi,q)$ so that $\phi$ is $\Pi_N$ and $\phi^\M \leq q$ (resp. $\phi^\M < q$).  The closed and open $\Sigma_N$ diagrams are defined similarly.
\end{enumerate}
\end{definition}

We now define the computable wff's of $L_{\omega_1\omega}$ and their codes by effective transfinite induction.  We follow the development of the classical case in \cite{Ash.Knight.2000}.  We presume an effective enumeration of the quantifier-free wff's of $L$.  We also presume effective codings of the following.
\begin{enumerate}
	\item All pairs of the form $(j, \overline{z})$, where $j \in \N$ and $\overline{z}$ is a tuple of variables.
	
	\item All quadruples of the form $(X, a, \overline{x}, e)$, where $X \in \{\Sigma, \Pi\}$, $a,e \in \N$, and $\overline{x}$ is a tuple of variables.
\end{enumerate}
When $\xi$ is a tuple of either of the above types, we let $\overline{\xi}$ denote the code of $\xi$.  

For every $X\in \{\Sigma,\Pi\}$ and $a\in \mathcal{O}$, we first define the index set $S_a^X$ in such a way that if $a\in\langle\alpha\rangle$, then every formula with indices in $S_a^X$ will be $X_\alpha$.  

We begin by setting $S_1^\Sigma$ and $S_1^\Pi$ to be the set of codes of all quantifier-free, finitary formulas of $\mathcal{L}$.  (Recall that $1$ denotes $0$ in Kleene's $\KO$.)  For every $a \in \KO - \{1\}$ and $X\in\{\Sigma,\Pi\}$, let $S_a^X$ be the set of codes of all quadruples of the form $(X,a,\overline{x},e)$, where $\overline{x}$ is a finite tuple of variable symbols, and $e\in\N$.  

Now for every $a\in\mathcal{O}$, $X\in\{\Sigma,\Pi\}$, and tuple of variable symbols $\overline{x}$, we define $P(X,a,\overline{x})$ to be the set of all codes of pairs $(j,\overline{z})$, where 
$j$ codes a quadruple $(X,b,\overline{y},e')$ with $b<_\mathcal{O} a$ and $\overline{z}$ is a finite sequence of variable symbols of $\overline{y}$ not contained in $\overline{x}$.

For each $i \in S^\Sigma_a\cup S^\Pi_a$, we define an infinitary wff $\phi_i$ as follows:
\begin{enumerate}
	\item If $a = 1$, then $\phi_i$ is the quantifier-free finitary wff indexed by $i$.  
	
	\item Suppose $a > 1$ and $i = \overline{(X,a,\overline{x},e)}$.  
	\begin{enumerate}
		\item If $X = \Sigma$, then 
		\[
		\phi_i = \infconj_{\overline{(j,\overline{z})} \in W_e\cap P(\Pi,a,\overline{x})} \inf_{\overline{z}} \phi_j.
		\]
	
		\item If $X = \Pi$, then 
		\[
		\phi_i = \infdis_{\overline{(j,\overline{z})} \in W_e\cap P(\Sigma,a,\overline{x})} \sup_{\overline{z}} \phi_j.
		\]
	\end{enumerate}
\end{enumerate}

For every computable ordinal $\alpha$, we let $\Sigma^c_\alpha$ denote the set of all formulas $\phi_i$ where $i\in\bigcup_{a\in\langle \alpha\rangle}S^\Sigma_a$.  Similarly, $\Pi^c_\alpha$ denotes the set of all formulas $\phi_i$ where $i\in\bigcup_{a\in\langle \alpha\rangle}S^\Pi_a$.   If $\psi =\phi_i$, then we say that $i$ is a \emph{code} of $\psi$.  By a \emph{computable infinitary formula}, we mean an element of $\Sigma_\alpha^c\cup\Pi_\alpha^c$ for some computable ordinal $\alpha$.

It is fairly routine to verify that all logical operations can be performed effectively via this coding system.  
For example, from an $i$ that codes an infinitary wff $\phi$, it is possible to compute a code of $\sup_x \phi$.  

\subsection{Combinatorial preliminaries}

We introduce here some results that will support our demonstration of lower bounds.  
Among these, our main result (Theorem \ref{thm:encode}) is a principle for representing $\Sigma^0_N$ and $\Pi^0_N$ sets as solutions of inequalities 
involving infinite series.  We believe this connection is sufficiently novel to merit consideration on its own.

We begin with the following lemma which is easily verified by simultaneous induction on $N$.  Note that the suprema and infima range over $\N$.

\begin{lemma}\label{lm:inf.sup}
For $R \subseteq \N^{N+1}$ and $n \in \N$, we have:
\begin{enumerate}
	\item $n \in \vec{\forall} R$ if and only if 
$\inf_{x_1} \sup_{x_2} \ldots Q_{x_N} \chi_R(n, x_1, \ldots, x_N) = 1$.  

	\item $n \in \vec{\exists} R$ if and only if 
	$\sup_{x_1} \inf_{x_2} \ldots Q_{x_N} \chi_R(n,x_1, \ldots, x_N) = 1$.
\end{enumerate}
\end{lemma}

To state our main theorem of this section, we need the following.

\begin{definition}
For $K,N \in \N$ and $f : \N^{N + 1} \rightarrow \R$ a bounded function, set: 
\begin{eqnarray*}
\Gamma_K(f; x_1, \ldots, x_N) & = & \sum_{x_0 = 0}^K 2^{-(x_0 + 1)} f(x_0, \ldots, x_N)\\
\Gamma(f; x_1, \ldots, x_N) & = & \sum_{x_0 = 0}^\infty 2^{-(x_0 + 1)} f(x_0, \ldots, x_N).\\
\end{eqnarray*}
\end{definition}

We define $\Gamma(f):\N^N\to \R$ by setting $\Gamma(f)(x_1,\ldots,x_N) = \Gamma(f;x_1,\ldots,x_N)$. 
 We note that $\Gamma(f)$ is computable if $f$ is computable and, in this case, an index of $\Gamma(f)$ can be computed from an index of $f$ and a bound on $f$.

We are now ready to state and prove the key result of this section.  In what follows, we view elements of $\N^{N+2}$ as being of the form $(x_0,x_1,\ldots,x_N,n)$.

\begin{theorem}\label{thm:encode}
Let $R \subseteq \N^{N + 2}$, and let $n \in \N$.
\begin{enumerate}
	\item $n \in \vec{\forall} R$ if and only if \label{thm:encode::forall}
	\[
	\inf_{x_1} \sup_{x_2} \ldots Q_{x_N} \Gamma(1 - \frac{1}{2} \chi_{R^*}; x_1, \ldots, x_N, n) \leq \frac{1}{2}.
	\]
	
	\item $n \in \vec{\exists} R$ if and only if \label{thm:encode::exists}
	\[
	\sup_{x_1} \inf_{x_2} \ldots Q_{x_N} \Gamma(\frac{1}{2} \chi_{(\neg R)^*}; x_1, \ldots, x_N, n) < \frac{1}{2}.
	\]
\end{enumerate}
\end{theorem}

The proof of the previous theorem requires a few preparatory lemmas.  For the first lemma, note that if $f:\N\to \mathbb R$ is a bounded function, then $\Gamma_K(f)$ is simply a real number (i.e. a constant).

\begin{lemma}\label{lm:carry}
If $f:\N \rightarrow \{\frac{1}{2},1\}$, then for every $K \in \N$,
$\Gamma_K(f) \leq \frac{1}{2}$ if and only if $f(m) = \frac{1}{2}$ for all $m < K$.
\end{lemma}

\begin{proof}[Proof sketch]
 Fix $K \in \N$. Consider the given sum in base $2$. Any $m < K$ for which $f(m)=1$ leads to a `carry' operation so that the $\frac{1}{2}$-position becomes $1$. Adding $f(K)$ would then force the value to be greater than $\frac{1}{2}$. 
 \end{proof}

\begin{lemma}\label{lm:forall.R.*}
Suppose $R \subseteq \N^{N+1}$.  Then $\vec\forall(R^*) = \vec{\forall} R$.
\end{lemma}

\begin{proof}[Proof sketch]
The proof that $\vec\forall(R^*) \subseteq \vec\forall R$ is straightforward.  
The other inclusion is demonstrated via Skolemization.
\end{proof}

\begin{lemma}\label{lm:swap}
Fix $R \subseteq \N^{N + 2}$ and $1 \leq J \leq N$.  Then for every $x_1,\ldots,x_{J-1},n\in \N$ and every $K\in \N$, we have:  
\begin{enumerate}
	\item $\sup_{x_J} \inf_{x_{J+1}} \ldots Q_{x_N} \Gamma_K (1 - \frac{1}{2} \chi_{R^*}; x_1, \ldots, x_N, n) \leq \frac{1}{2}$ if and only if\\ $
	\Gamma_K(\sup_{x_J} \inf_{x_{J+1}} \ldots Q_{x_N} (1 - \frac{1}{2} \chi_{R^*}); x_1, \ldots, x_{J-1}, n) \leq \frac{1}{2}$.\label{lm:swap::sup}
	
	\item $\inf_{x_J} \sup_{x_{J+1}} \ldots Q_{x_N} \Gamma_K (1 - \frac{1}{2} \chi_{R^*}; x_1, \ldots, x_N, n) \leq \frac{1}{2}$ if and only if \\
	$\Gamma_K(\inf_{x_J} \sup_{x_{J+1}} \ldots Q_{x_N} (1-\frac{1}{2} \chi_{R^*}; x_1, \ldots, x_{J-1}, n) \leq \frac{1}{2}$.\label{lm:swap::inf}
\end{enumerate}
\end{lemma}

\begin{proof}
Set $G = 1 - \frac{1}{2}\chi_{R^*}$ and note that $\ran(G) \subseteq \{\frac{1}{2}, 1\}$.  Thus, 
in what follows, all suprema are maxima and all infima are minima.  Also, we may assume $K > 0$.  

We proceed by induction on $N - J$.  We begin with the base case for (\ref{lm:swap::sup}), that is, $J=N-1$.  
Without loss of generality, we may assume that one of the two quantities in (\ref{lm:swap::sup}) is no larger than $\frac{1}{2}$.  
Since $\Gamma_K(\sup_{x_N} G; x_1, \ldots, x_{N - 1}, n) \geq \sup_{x_N} \Gamma_K(G; x_1, \ldots, x_N, n)$, 
we may assume $\sup_{x_N} \Gamma_K(G; x_1, \ldots, x_N, n) \leq \frac{1}{2}$. 
By Lemma \ref{lm:carry}, we have that $G(x_0, x_1, \ldots, x_N, n)  = \frac{1}{2}$
for all $x_N \in \ N$ and all $x_0 < K$.  
By Lemma \ref{lm:carry} again, $\Gamma_K(\sup_{x_N} G; x_1, \ldots, x_{N - 1}, n) \leq \frac{1}{2}$.

We now consider the base case for (\ref{lm:swap::inf}).  Again, we may assume one of the two quantities in (\ref{lm:swap::inf}) is no larger than $\frac{1}{2}$.  Since $\Gamma_K(\inf_{x_N} G; x_1, \ldots, x_{N - 1}, n) \leq \inf_{x_N} \Gamma_K(G; x_1, \ldots, x_N, n)$, we assume $\Gamma_K(\inf_{x_N} G; x_1, \ldots, x_{N -1}, n) \leq \frac{1}{2}$.
By Lemma \ref{lm:carry}, $\inf_{x_N} G(x_0, \ldots, x_N, n) = \frac{1}{2}$ for all $x_0 < K$.  Consequently, for each $x_0 < K$, there exists 
$\xi_{x_0} \in \N$ so that $G(x_0, \ldots, x_{N-1},\xi_{x_0},n) = \frac{1}{2}$.  
Let
\[
\xi = \left\{
\begin{array}{cc}
\max_{x_0<K} \xi_{x_0} & \mbox{if $N$ odd}\\
0 
& \mbox{otherwise}.\\
\end{array}
\right.
\]
By the definition of $R^*$, it follows that $G(x_0, \ldots, x_{N-1},\xi,n) = \frac{1}{2}$ for all $x_0 < K$.  By Lemma \ref{lm:carry} again, 
$\inf_{x_N} \Gamma_K(G; x_1, \ldots, x_N, n) \leq \frac{1}{2}$.

We now perform the inductive step for (\ref{lm:swap::sup}).  Suppose that $N-J>1$ and set $H = \inf_{x_{J+1}} \ldots Q_{x_N} G$.  
By the inductive hypothesis, it suffices to show that 
$\sup_{x_J} \Gamma_K(H; x_1, \ldots, x_J, n) \leq \frac{1}{2}$ if and only if $\Gamma_K(\sup_{x_J} H; x_1, \ldots, x_{J - 1}, n) \leq \frac{1}{2}$. 
Without loss of generality, we assume $\sup_{x_J} \Gamma_K(H; x_1, \ldots, x_J, n) \leq \frac{1}{2}$
By Lemma \ref{lm:carry}, for all $x_J \in \ N$ and all $x_0 < K$, 
$H(x_0, x_1, \ldots, x_J, n) = \frac{1}{2}$.  
By Lemma \ref{lm:carry} again, 
$\Gamma_K(\sup_{x_J} H ; x_1, \ldots, x_{J - 1}, n) \leq \frac{1}{2}$.

We now carry out the inductive step for (\ref{lm:swap::inf}). 
In this case, we consider the function $H = \sup_{x_{J+1}} \ldots Q_{x_N} G(x_0, \ldots, x_N, n)$.  
It suffices to show that $\inf_{x_J} \Gamma_K(H; x_1, \ldots, x_J, n) \leq \frac{1}{2}$ if and only if 
$\Gamma_K (\inf_{x_J} H; x_1, \ldots, x_{J - 1}, n) \leq \frac{1}{2}$. 
Without loss of generality, we assume $\Gamma_K (\inf_{x_J} H; x_1, \ldots, x_{J - 1}, n) \leq \frac{1}{2}$.  By Lemma \ref{lm:carry}, for every $x_0 < K$, 
$\inf_{x_J} H(x_0, \ldots, x_J, n) = \frac{1}{2}$, 
whence, for every $x_0 < K$, there exists $\xi_{x_0} \in \N$ so that $H(x_0, \ldots, x_{J-1},\xi_{x_0}, n) = \frac{1}{2}$.  
Let 
\[
\xi = \left\{
\begin{array}{cc}
\max_{x_0 < K} \xi_{x_0} & \mbox{$J$ odd} \\
0 
& \mbox{otherwise}. \\
\end{array}
\right.
\]
By the definition of $R^*$, $H(x_0, x_1, \ldots, x_{J-1},\xi, n) = \frac{1}{2}$ for all $x_0 < K$.  
By Lemma \ref{lm:carry}, 
$\Gamma_K(\inf_{x_J} H; x_0, \ldots, x_{J - 1}, n) = \frac{1}{2}$.  
\end{proof}

We note that while Lemma \ref{lm:swap} is hardly the key result of this section, it is nevertheless somewhat surprising.  
In general, one does not expect to be able to interchange summation with $\sup$ or $\inf$.  
It is here that the use of $R^*$ comes in to consideration and provides a path to a weaker conclusion but one that is 
just strong enough to effect the rest of the proof.

\begin{proof}[Proof of Theorem \ref{thm:encode}]
It suffices to prove (\ref{thm:encode::forall}); part (\ref{thm:encode::exists}) follows by considering complements.
Once again, set $G = 1 - \frac{1}{2} \chi_{R^*}$. 

Suppose $n \in \vec{\forall} R$.  It follows from Lemmas \ref{lm:inf.sup} and \ref{lm:forall.R.*}  that 
$$\sup_{x_0} \inf_{x_1} \ldots Q_{x_N} G(x_0, \ldots, x_N, n) = \frac{1}{2}.$$  
Thus, by Lemma \ref{lm:carry}, 
$\Gamma_K(\inf_{x_1} \ldots Q_{x_N} G; n) \leq \frac{1}{2}$.  
By Lemma \ref{lm:swap}, we have that
$$\inf_{x_1} \ldots Q_{x_N} \Gamma_K(G; x_1, \ldots, x_N, n) \leq \frac{1}{2}.$$  
Since $G \leq 1$, it follows that 
\[
\inf_{x_1} \ldots Q_{x_N} \Gamma(G; x_1, \ldots, x_N, n) \leq \frac{1}{2} + 2^{-(K+1)}
\]
 for all $K \in \N$.
Hence, $\sup_{x_0} \inf_{x_1} \ldots Q_{x_N} \Gamma(G; x_1, \ldots, x_N, n) \leq \frac{1}{2}$. 

Conversely, suppose $\inf_{x_1} \sup_{x_2} \ldots Q_{x_N} \Gamma(G; x_1, \ldots, x_N, n) \leq \frac{1}{2}$.
Since $G > 0$, for every $K \in \N$, 
$\inf_{x_1} \sup_{x_2} \ldots Q_{x_N} \Gamma_K(G; x_1, \ldots, x_N, n) \leq \frac{1}{2}$.
By Lemmas \ref{lm:carry} and \ref{lm:swap}, for every $x_0 < K$, 
$\inf_{x_1} \sup_{x_2} \ldots Q_{x_N} G(x_0, \ldots, x_N, n) = \frac{1}{2}$.  
Thus, $\sup_{x_0} \inf_{x_1} \sup_{x_2} \ldots Q_{x_N} G(x_0, \ldots, x_N, n) = \frac{1}{2}$. 
It follows from Lemma \ref{lm:inf.sup} that $n \in \vec{\forall} R^*$.  Thus, by Lemma \ref{lm:forall.R.*}, 
$n \in \vec{\forall} R$. 
\end{proof}

\section{Finitary diagram results- upper bounds}\label{sec:fin.up}

We begin by considering the quantifier-free diagrams.

\begin{proposition}\label{prop:atomic}
If $\M$ is a computably presentable $L$-structure, then the closed quantifier-free diagram of $\M$ is $\Pi^0_1$ and the 
open quantifier-free diagram of $\M$ is $\Sigma^0_1$.
\end{proposition}

\begin{proof}
The proposition follows from the observation that if $\M$ is computably presentable, then 
the map $\phi \mapsto \phi^\M$ is computable on the set of quantifier-free sentences of $L$.
\end{proof}

We note that the proof of Proposition \ref{prop:atomic} is uniform; that is, from an index of a presentation of $\M$, it is possible to compute a $\Pi^0_1$ index of the closed quantifier-free diagram of $\M$ and a $\Sigma^0_1$ index of the open quantifier-free diagram of $\M$.

We now consider the higher-level diagrams.

\begin{theorem}\label{thm:fin.up}
Let $\M$ be a computably presentable $L$-structure, and let $N$ be a positive integer.  
\begin{enumerate}
	\item The closed $\Pi_N$ diagram of $\M$ is $\Pi^0_N$, and the open $\Pi_N$ diagram of $\M$ is $\Sigma^0_{N+1}$.
	\label{thm:fin.up::pi}
	
	\item The closed $\Sigma_N$ diagram of $\M$ is $\Pi^0_{N+1}$, and the open $\Sigma_N$ diagram of $\M$ is $\Sigma^0_N$. \label{thm:fin.up::sigma}
\end{enumerate}
Moreover, the results of (\ref{thm:fin.up::pi}) and (\ref{thm:fin.up::sigma}) hold uniformly in the sense that from $N$ and an index for a computable presentation for $\M$, one can compute an index for any of the above diagrams.
\end{theorem}

\begin{proof} Throughout this proof, we fix a computable presentation $\M^\sharp$ of $\M$.  We proceed by induction on $N$, the base case being true by Proposition \ref{prop:atomic}.  We now fix a positive integer $N$ and assume that  (\ref{thm:fin.up::pi}) and (\ref{thm:fin.up::sigma}) hold uniformly for every $M < N$.  

Fix a $\Pi_N$ sentence $\phi$ and a rational number $q$.  Note that $\phi$ has the form $\sup_{\overline{x}}\psi$, where $\psi$ is a $\Sigma_{N-1}$ wff of $L$ and $\overline{x}$ is a tuple of variables.
 Since the rational points of $\M^\sharp$ are dense, $\sup_{\overline{a}\in \mathbb{Q}(\M^\sharp)}\psi^\M(\overline{a})=\sup_{\overline{a}\in |\M|}\psi^\M(\overline{a})$. 
Thus,
\begin{eqnarray*}
\phi^\M\leq q 
	&\iff & (\forall k\in\N) \ (\forall\overline{a}\in\mathbb{Q}(\M^\sharp)) \ \psi^\M(\overline{a})\leq q+2^{-k}.
\end{eqnarray*}

If $N = 1$, then by the uniformity of Proposition \ref{prop:atomic}, the statement $\psi^\M(\overline{a})\leq q+2^{-k}$ is a $\Pi^0_1$ condition on $\phi,\overline{a},k$.  If $N > 1$, then this statement is a $\Pi^0_N$ condition since (\ref{thm:fin.up::sigma}) is assumed to hold uniformly for $M<N$. In either case it then follows that $\phi^\M\leq q$ is a $\Pi^0_N$ condition on $\phi,q$. 

Furthermore,
\begin{eqnarray*}
\phi^\M< q  
	& \iff & (\exists k\in\N) \ (\forall \overline{a}\in\mathbb{Q}(\M^\sharp)) \ \psi^\M(\overline{a})\leq q-2^{-k}.
\end{eqnarray*}

As before, if $N = 1$, then the statement $\psi^\M(\overline{a})\leq q-2^{-k}$ is a $\Pi^0_1$ condition on $\phi,\overline{a},k$. If $N > 1$, then this statement is a $\Pi^0_N$ condition since (\ref{thm:fin.up::sigma}) is assumed to hold uniformly for $M<N$. In either case, it follows that $\phi^\M< q$ is a $\Sigma^0_{N+1}$ condition on $\phi,q$.

Now fix a $\Sigma_N$ sentence $\phi$ and a rational number $q$. Then $\phi$ has the form $\inf_{\overline{x}}\psi$, where $\psi$ is a $\Pi_{N-1}$ wff of $L$ and $\overline{x}$ is a tuple of variables. Again, since the rational points of $\M^\sharp$ are dense, $\inf_{\overline{a}\in \mathbb{Q}(\M^\sharp)}\psi^\M(\overline{a})=\inf_{\overline{a}\in |\M|}\psi^\M(\overline{a})$. Thus,
\begin{eqnarray*}
\phi^\M\leq q& \iff 
& (\forall k\in\N) \  (\exists \overline{a}\in\mathbb{Q}(\M^\sharp)) \ \psi^\M(\overline{a})< q+2^{-k}.
\end{eqnarray*}
If $N = 1$, then the statement $\psi^\M(\overline{a})< q+2^{-k}$ is a $\Sigma^0_1$ condition on $\phi,\overline{a},k$. If $N > 1$, then this statement is a $\Sigma^0_N$ condition since (\ref{thm:fin.up::pi}) is assumed to hold uniformly for $M<N$. In either case, it then follows that $\phi^\M\leq q$ is a $\Pi^0_{N+1}$ condition on $\phi,q$. 

Finally,
\begin{eqnarray*}
\phi^\M< q 
	&\iff & (\exists k\in\N) \ (\exists \overline{a}\in\mathbb{Q}(\M^\sharp)) \ \psi^\M(\overline{a})< q-2^{-k}.
\end{eqnarray*}
If $N = 1$, then the statement $\psi^\M(\overline{a})< q-2^{-k}$ is a $\Sigma^0_1$ condition on $\phi,\overline{a},k$. If $N > 1$, then this statement is a $\Sigma^0_N$ condition since (\ref{thm:fin.up::pi}) is assumed to hold uniformly for $M<N$. In either case, it then follows that $\phi^\M< q$ is a $\Sigma^0_N$ condition on $\phi,q$.  

Finally, we note that these arguments are uniform in the sense described above. 
\end{proof}

\section{Finitary diagram results- lower bounds}\label{sec:fin.lwr}

We demonstrate that the results in Section \ref{sec:fin.up} are the best possible by means of the following.

\begin{theorem}\label{thm:fin.lwr}
There is a language $L'$ and a computably presentable $L'$-structure $\M$ with the following properties:
\begin{enumerate}
	\item The closed quantifier-free diagram of $\M$ is $\Pi^0_1$-complete, and the open quantifier-free diagram of $\M$ is $\Sigma^0_1$-complete.
	
	\item For every positive integer $N$, the closed $\Pi_N$ diagram of $\M$ is $\Pi^0_N$-complete, and the open $\Pi_N$ diagram of $\M$ is $\Sigma^0_{N+1}$-complete.
	
	\item For every positive integer $N$, the closed $\Sigma_N$ diagram of $\M$ is $\Pi^0_{N+1}$-complete, and the open $\Sigma^0_N$ diagram of $\M$ is $\Sigma^0_N$-complete.
\end{enumerate}
\end{theorem}

\begin{proof}
Let $L'$ be the metric language that consists of the following.
\begin{enumerate}
	\item A constant symbol $\underline{0}$.

	\item A family of constant symbols $(c_n)_{n \in \N}$.  
	
	\item A family of predicate symbols $(P_{N,n})_{N,n \in \N}$, where $P_{2N, n}$ and $P_{2N + 1, n}$ are $(N+1)$-ary.
\end{enumerate}
Here, each predicate symbol is assumed to have modulus of continuity equal to the constant function $1$.

We now define our $L'$-structure $\M$.  The underlying metric space of $\M$ is the set $\N$ of natural numbers equipped with its discrete metric.  We also set $\underline{0}^\M =0$.  In order to define the interpretations of the other symbols, we first set 
\[
f_N = \left\{
\begin{array}{cc}
\Gamma(1 - \frac{1}{2} \chi_{R_N^*}) & \mbox{$N$ even}\\
\Gamma(\frac{1}{2} \chi_{(\neg R_N)^*}) & \mbox{otherwise.}\\
\end{array}
\right.
\]
We can now set
\[
c_n^\M = \left\{
\begin{array}{cc}
f_0(n/2) & \mbox{$n$ even}\\
f_1((n - 1)/2) & \mbox{otherwise}.\\
\end{array}
\right.
\]
Finally, set $P^\M_{2N,n}(a_0, \ldots, a_N) = f_{2N + 2}(a_0, \ldots, a_N, n)$, and let 
$P^\M_{2N + 1, n}(a_0, \ldots, a_N)  = f_{2N + 3}(a_0, \ldots, a_N, n)$.

It is clear that $\M$ has a computable presentation.  In fact, one may simply take the $n$-th distinguished point to be 
$n$.

We first note that the closed atomic diagram of $\M$ is $\Pi^0_1$-complete.
To see this, let $\phi_n$ be the sentence $d(c_{2n},0)$.  Then, by Theorem \ref{thm:encode},
$\phi_n^\M \leq \frac{1}{2}$ if and only if $n \in \vec{\forall} R_0$.  

Similarly, the open atomic diagram of $\M$ is $\Sigma^0_1$-complete.  
This time, let $\phi_n$ be the sentence $d(c_{2n+1}, 0)$.  Then, by Theorem \ref{thm:encode}, 
$\phi_n^\M < \frac{1}{2}$ if and only if $n \in \vec{\exists} R_0$.  

Next fix a positive integer $N$.   For each $n \in \N$, let $\phi_n$ be the sentence
$$\inf_{x_1} \ldots Q_{x_N} P_{2N, n}(x_1, \ldots, x_N),$$ and let 
$\psi_n$ be the sentence $$\sup_{x_1} \ldots Q_{x_N} P_{2N + 1, n}(x_1, \ldots, x_N).$$  
By Theorem \ref{thm:encode}, $\phi_n^\M \leq \frac{1}{2}$ if and only if $n \in \vec{\forall} R_{2N}$.  
Thus, the closed $\Sigma_N$ diagram of $\M$ is $\Pi^0_{N+1}$-complete.
Also by Theorem \ref{thm:encode}, $\psi_n^\M < \frac{1}{2}$ if and only if $n \in \vec{\exists} R_{2N + 1}$.
Thus, the open $\Pi_N$ diagram of $\M$ is $\Sigma^0_{N + 1}$-complete.

Since the open $\Pi_{N - 1}$ diagram of $\M$ is $\Sigma^0_N$-complete, it follows that the 
open $\Sigma_N$ diagram of $\M$ is $\Sigma^0_N$-complete.   
It similarly follows that the closed $\Pi_N$ diagram of $\M$ is $\Pi^0_N$-complete.
\end{proof}

We conclude this section with some remarks on the choice of structure in the above proof.  
Since structures in continuous logic must be bounded, it might seem that the unit interval is a 
natural setting in which to construct these lower bounds.  However, it is well-known that the evaluation of maxima of computable functions on a computably compact space is a computable operation (see, e.g. Chapter 6 of \cite{Weihrauch.2000}).  Thus, the closed and open diagrams for a metric structure with a computably compact presentation are $\Pi^0_1$ and $\Sigma^0_1$ respectively.  It is fairly easy to see that the standard presentation of $[0,1]$ (i.e. the presentation in which the distinguished points are precisely the rational numbers in $[0,1]$) is computably compact.  
On the other hand, the natural numbers under the discrete metric provides the simplest non-trivial setting that is bounded and not compact.  

\section{Infinitary results}\label{sec:inf}

When formulating our diagram complexity results for infinitary logic, we actually must eschew the terminology of diagrams.  The reason for this is that, because of the coding of the computable infinitary formulae, these diagrams are capable of computing $\KO$, which itself is $\Pi^1_1$-complete.  In order to avoid this pitfall, we focus on the complexity of the right Dedekind cuts of reals of the form $\phi^\M$ where $\phi$ is infinitary.  To this end, for $x \in \R$, we let $D^>(x)$ denote the right Dedekind cut of $x$, that is, 
\[
D^{>}(x) = \{q \in \Q\ :\ q > x\}.
\]
We also set 
\[
D^{\geq}(x) = \{q \in \Q\ :\ q \geq x\}.
\]
Of course, if $x$ is irrational, then $D^>(x) = D^{\geq}(x)$.  In terms of evaluating complexity, differences only arise when considering uniformity.

We first prove our infinitary upper bound result which generalizes our bounds in the finitary case.

\begin{theorem}\label{thm:inf.up}
Let $\M$ be a computably presentable $L$-structure and let $\phi$ be a computable infinitary sentence of $L$.  
\begin{enumerate}
	\item If $\phi$ is $\Pi^c_\alpha$, then $D^>(\phi^\M)$ is $\Sigma^0_{\alpha+1}$ uniformly in a code of $\phi$, and 
	$D^\geq(\phi^\M)$ is $\Pi^0_\alpha$ uniformly in a code of $\phi$.\label{thm:inf.up::pi}  
	
	\item If $\phi$ is $\Sigma^c_\alpha$, then $D^>(\phi^\M)$ is $\Sigma^0_\alpha$ uniformly in a code of $\phi$, and $D^\geq(\phi^\M)$ is $\Pi^0_{\alpha + 1}$ uniformly in a code of $\phi$.\label{thm:inf.up::sigma}
\end{enumerate}
\end{theorem}

\begin{proof}
 Fix a computable presentation $\M^\sharp$ of $\M$.  Let $\phi$ be a computable infinitary sentence of $L$.

Suppose $\phi \in \Sigma^c_\alpha \cup \Pi^c_\alpha$.   A code for $\phi$ yields a notation $a$ for $\alpha$.  In the following, all other ordinals considered are less than $\alpha$.  
For ease of exposition, we identify each $\beta \leq \alpha$ with its unique notation in $\{b\ :\ b \leq_\KO a\}$.  

We proceed by effective transfinite recursion.  Thus, we assume the following hold uniformly in an index of $\M^\#$.

\begin{enumerate}
	\item From a $\beta < \alpha$ and a code of a $\Pi_\beta^c$ sentence $\psi$, it is possible to compute a $\Pi^0_\beta$ index of $D^\geq(\psi^\M)$ and a $\Sigma^0_{\beta + 1}$-index of 
$D^>(\psi^\M)$.
	
	\item From a $\beta < \alpha$ and a code of a $\Sigma_\beta^c$ sentence $\psi$, it is possible to compute
	a $\Pi^0_{\beta + 1}$-index of $D^\geq(\psi)$ and a $\Sigma^0_\beta$-index of $D^>(\psi_i)$.
\end{enumerate}

First suppose that $\phi$ is a $\Pi_\alpha^c$ sentence.  Thus, $\phi$ has the form 
$\infdis_{i \in I} \sup_{\overline{x_i}}\phi_i$ where $I$ is c.e. and $\phi_i$ is $\Sigma_{\beta_i}^c$ for some $\beta_i < \alpha$.  Furthermore, we may assume $(\beta_i)_{i \in I}$ is computable.  
For $q \in \Q$, we have
\[
q \in D^\geq(\phi^\M) \Leftrightarrow (\forall k \in \N)( \forall i \in I)( \forall \overline{r} \in \Q(\M^\#))\ q + 2^{-k} \in D^>(\phi_i^\M(\overline{r})).
\]
As $\emptyset^{(\alpha)}$ computes $D^>(\Phi_i^\M(\overline{r}))$ uniformly in $i$, $D^\geq(\phi^\M)$ is 
co-c.e. in $\emptyset^{(\alpha)}$, that is, $D^\geq(\phi^\M)$ is $\Pi^0_\alpha$.
At the same time, 
\[
q \in D^>(\phi^\M)\ \iff\ (\exists k \in \N)( \forall i \in I)( \forall \overline{r} \in \Q(\M^\#))\ q - 2^{-k} \not \in D^>(\phi_i^\M(\overline{r}).
\]
Thus, $D^>(\phi^\M)$ is $\Sigma^0_2(\emptyset^{(\alpha)}) = \Sigma_{\alpha + 1}^0$.  

Now suppose $\phi$ is a $\Sigma^c_\alpha$ sentence.  Thus, $\phi$ has the form 
$\infconj_{i \in I} \inf_{\overline{x_i}}\phi_i$ where $I$ is c.e. and $\phi_i$ is $\Pi_{\beta_i}^c$ for some $\beta_i < \alpha$ uniformly in $i$.    
Let $q \in \Q$.  Then, 
\[
q \in D^\geq(\phi^\M)\ \iff\ (\forall k \in \N)( \exists i \in I)( \exists \overline{r} \in \Q(\M^\#)) q + 2^{-k} \in D^>(\phi_i^\M(\overline{r})).
\]
Thus, $D^\geq(\phi^\M)$ is $\Sigma^0_2(\emptyset^{(\alpha)}) = \Sigma^0_{\alpha + 1}$.
In addition, 
\[
q \in D^>(\phi_{i_0}^\M)\ \iff\ (\exists k \in \N)( \exists i \in I)( \exists \overline{r} \in \Q(\M^\#))\ q - 2^{-k} \not \in D^>(\phi_i^\M(\overline{r})).
\]
Thus, $D^>(\phi^\M)$ is $\Sigma^0_1(\emptyset^{(\alpha)}) = \Sigma^0_\alpha$.

As these arguments are all uniform in an index of $\M^\#$ and a code for $\phi$, the theorem is proven.
\end{proof}

We now demonstrate the optimality of Theorem \ref{thm:inf.up} by means of the following.

\begin{theorem}\label{thm:inf.lwr}
There is a language $L''$ and an $L''$-structure $\M$ so that the following hold for every computable ordinal $\alpha$.
\begin{enumerate}
	\item There is a computable sequence $(\psi_i)_{i \in \N}$ of $\Pi^c_\alpha$ sentences of $L''$ so that
	$\{i\ :\ \frac{1}{2} \in D^\geq(\psi^\M_\alpha)\}$ is $\Pi^0_\alpha$-complete.\label{thm:inf.lwr::pi.>=}
	
	\item There is a computable sequence $(\psi_i)_{i \in \N}$ of $\Sigma^c_\alpha$ sentences of $L''$ so that 
	$\{i\ :\ \frac{1}{2} \in D^>(\psi^\M_\alpha)\}$ is $\Sigma^0_\alpha$-complete.\label{thm:inf.lwr::sigma.>}
	
	\item There is a computable sequence $(\psi_i)_{i \in \N}$ of $\Pi^c_\alpha$ sentences of $L''$ so that 
	$\{i\ :\ \frac{1}{2} \in D^>(\psi^\M_\alpha)\}$ is $\Sigma^0_{\alpha + 1}$-complete.\label{thm:inf.lwr::pi.>}
	
	\item There is a computable sequence $(\psi_i)_{i \in \N}$ of $\Sigma^c_\alpha$ sentences of $L''$ so that
	$\{i\ :\ \frac{1}{2} \in D^\geq(\psi^\M_\alpha)\}$ is $\Pi^0_{\alpha+ 1}$-complete.\label{thm:inf.lwr::sigma.>=}
\end{enumerate}  
\end{theorem}

The remainder of this section is dedicated to the proof of Theorem \ref{thm:inf.lwr}.  We begin with the construction 
of $L''$ and $M''$.

Let $L_0$ be a language consisting of one constant symbol $\underline{q}$ for every $q\in \Q\cap[0,1]$ and let $\M_0$ be the $L_0$-structure whose underlying metric space is $[0,1]$ with its usual metric and which interprets each $\underline{q}$ as $q$.  Let $L''$ be the expansion of $L_0$ obtained by adding a family $(c_{N, n, x_1, \ldots, x_{N + 1}})_{N, n, x_1, \ldots, x_{N+1} \in \N}$
of constant symbols.  

Let $\M$ be the expansion of $\M_0$ obtained by setting 
$c_{N, n, x_1, \ldots, x_{N+1}}^\M = \frac{1}{2}(1 - \chi_{R_{2N+1}}(n,x_1, \ldots, x_{N+1}))$.
Since $(R_N)_{N \in \N}$ is computable, it follows that $\M$ is computably presentable.

We now verify that $L''$ and $\M$ satisfy the conclusions of Theorem \ref{thm:inf.lwr}.  We will need a little additional terminology and two lemmas.

Suppose $(\psi_i)_{i \in \N}$ is a sequence of $\Pi^c_\alpha$ sentences of $L''$.  
We say that a set $S$ is \emph{encoded} by $(\psi_i)_{i \in \N}$ if $\psi_i^\M = 1 - \frac{1}{2} \chi_S(i)$ for all $i$.

Similarly, if $(\psi_i)_{i \in \N}$ is a sequence of $\Sigma^c_\alpha$ sentences of $L''$, 
we say that a set $S$ is \emph{encoded} by $(\psi_i)_{i \in \N}$ if $\psi_i^\M = \frac{1}{2}(1 - \chi_S(i))$ for all $i$.

\begin{lemma}\label{lm:encode}
Let $\alpha$ be a computable ordinal.
\begin{enumerate}
	\item Every $\Sigma^0_\alpha$ set is encoded by a computable sequence of $\Sigma^c_\alpha$ sentences.
	\label{lm:encode::sigma}

	\item Every $\Pi^0_\alpha$ set is encoded by a computable sequence of $\Pi^c_\alpha$ sentences.\label{lm:encode::pi}	
\end{enumerate}
\end{lemma}

\begin{proof}
We prove (\ref{lm:encode::sigma}).  Part (\ref{lm:encode::pi}) then follows by considering complements.  Suppose $S$ is $\Sigma^0_\alpha$.

If $\alpha = 0$, then we let 
\[
\psi_i = \left\{
\begin{array}{cc}
d(\underline{0},\underline{0}) & i \in S\\
d(\underline{0}, \underline{\frac{1}{2}}) & \mbox{otherwise.}\\
\end{array}
\right.
\]
Next suppose $\alpha = N + 1$ where $N \in \N$. 
Let
\[
\psi_n = \infconj_{x_1} \infdis_{x_2} \ldots \mathcal{C}_{x_{N+1}} d(c_{N, n, x_1, \ldots, x_{N+1}}, \underline{0}).
\] 
Here, $\mathcal{C}$ is $\infconj$ if $N$ is even and $\infdis$ if $N$ is odd.

It follows from Lemma \ref{lm:inf.sup} that $(\psi_n)_{n \in \N}$ encodes $\vec{\exists} R_{2N+1}$.  
Since $\vec{\exists} R_{2N + 1}$ is $\Sigma_{N + 1}^0$-complete, it follows that every $\Sigma^0_{N + 1}$ set is encoded by 
a sequence of computable $\Sigma_{N + 1}^c$ sentences.   Furthermore, the construction of such a sequence from a $\Sigma^0_{N+1}$ index is uniform.

Suppose $\alpha \geq\omega$.  Similar to the proof of Theorem 7.9 of \cite{Ash.Knight.2000}, we construct a sequence $(\phi_n)_{n \in \N}$ of $\Sigma^0_\alpha$ sentences so that $\phi_n^\M = 1 - \chi_S(n)$.  In particular, we replace 
$\true$ and $\false$ with $d(\underline{0},\underline{0})$ and $d(\underline{0},\underline{1})$ respectively.   Setting $\psi_n = \frac{1}{2} \phi_n$ yields the desired formulae.
\end{proof}

\begin{lemma}\label{lm:sum}
If $(\psi_n)_{n \in \N}$ is a computable sequence of $\Pi^c_\alpha$ sentences of $L''$, then there is a computable $\Pi^c_\alpha$ sentence $\phi$ of $L''$ so that 
\[
\phi^\M = \sum_{n = 0}^\infty 2^{-(n+1)} \psi_n^\M.
\]
Furthermore, a code of $\phi$ can be computed from an index of $(\psi_n)_{n \in \N}$.
\end{lemma}

\begin{proof}
For $a,b \in [0,1]$, let $\avg(a,b) = \frac{1}{2}(a + b)$.  By inspection, \\
\[
\avg(a,b) = \max\{a \intdiff \frac{1}{2} (a \intdiff b), b \intdiff \frac{1}{2} (b \intdiff a)\}.
\]
Thus, we may regard $\avg$ as a connective.  If $\phi, \psi$ are quantifier-free, then so is $\avg(\phi,\psi)$.

Since $\avg$ is increasing in each variable and continuous, it follows that 
$\avg(\sup_j a_j, \sup_k b_k) = \sup_{j,k} \avg(a_j,b_k)$ and 
$\avg(\inf_j a_j, \inf_k b_k) = \inf_{j,k} \avg(a_j,b_k)$.  From this it follows that $\avg(\phi,\psi)$ is equivalent to a $\Pi^c_\alpha$ (resp. $\Sigma^c_\alpha$) sentence if $\phi$ and $\psi$ are $\Pi^c_\alpha$ (resp. $\Sigma^c_\alpha$) sentences.

When $a_0, \ldots, a_{K + 1} \in [0,1]$, note that
\[
\sum_{n = 0}^{K + 1} 2^{-(n + 1)} a_n = \avg(\phi_0, \sum_{n = 0}^K 2^{-(n+1)} \phi_{n + 1}).
\]
Thus, we may regard inner product with $(2^{-(n + 1)})_{n = 0}^K$ as a connective.  
Furthermore, a code of $\sum_{n = 0}^K 2^{-(n + 1)} \phi_n$ can be computed from codes of $\phi_0, \ldots, \phi_K$.  

Finally, when $a_n \in [0,1]$, we have 
\[
\sum_{n = 0}^\infty a_n = \sup_K \sum_{n = 0}^K a_n.
\]
The conclusion of the lemma follows.
\end{proof}

\begin{proof}[Proof of Theorem \ref{thm:inf.lwr}]
Parts (\ref{thm:inf.lwr::pi.>=}) and (\ref{thm:inf.lwr::sigma.>}) follow directly from Lemma \ref{lm:encode}.  

Now suppose $S$ is $\Sigma^0_{\alpha+1}$ complete.  Take a $\Pi^0_\alpha$ binary relation $R$ so that 
$S = \vec{\exists} R$.  By Lemma \ref{lm:encode}, there is a computable family $(\psi_{n,x_1})_{n, x_1\in \N}$ 
of $\Pi^c_\alpha$ sentences so that for all $n,x_1 \in \N$, $\psi_{n,x_1}^\M = 1 - \frac{1}{2} \chi_R(n,x_1)$.
By Lemma \ref{lm:sum}, there is a computable sequence $(\phi_n)_{n \in \N}$ of $\Pi^c_\alpha$ sentences
so that 
\[
\phi_n^\M = \sum_{x_1 = 0}^\infty 2^{-(x_1 + 2)} \psi_{n,x_1}^\M.
\]
It then follows that $n \in S$ if and only if $\frac{1}{2} \in D^>(\phi_n^\M)$, establishing (\ref{thm:inf.lwr::pi.>}).  
Part (\ref{thm:inf.lwr::sigma.>=}) follows by considering complements. 
\end{proof}

Returning to an earlier point, we note that the closed and open quantifier-free diagrams of $\M$ are $\Pi^0_1$-complete and $\Sigma^0_1$-complete respectively.  
To see this, fix a $\Sigma^0_1$ complete set $C$, and let $(c_s)_{s = 0}^\infty$ be an effective enumeration of $C$. 
Since $C$ is infinite, we may assume this enumeration is one-to-one.  
Let
\[
p_n = 
\begin{cases}
\frac{1}{2} - 2^{-s} & \text{ if } n= c_s\\
\frac{1}{2} & \text{ otherwise. }
\end{cases}
\]
It is fairly straightforward to show that $(p_n)_{n \in \N}$ is computable as a sequence of reals.  Furthermore, 
$p_n < \frac{1}{2}$ if and only if $n \in C$.  Since $L''$ contains a constant symbol for each rational number, it follows that the open quantifier-free diagram of $\M$ is $\Sigma^0_1$-complete.  
The $\Pi^0_1$-completeness of the closed quantifier-free diagram follows by considering complements.

We also note that while computably compact domains are insufficient for demonstrating lower bounds in the finitary case, $[0,1]$ works swimmingly in the infinitary case.

Finally, we note that the infinitary sentences in the above proof are built up from quantifier-free sentences.  
Thus, they do not require moduli of continuity.  Therefore, although we have framed our work in an effectivization of the infinitary continuous logic of Eagle, our results will hold in any reasonable effectivization of the infinitary continuous logic of Ben-Yaacov and Iovino.
 
\section{Conclusion}\label{sec:conclusion}

We have introduced a framework for examining the complexity of the quantifier levels of the finitary and infintary theory of a computably presented metric space and we have pinned down the complexity at each level in terms of the hyperarithmetical hierarchy.  Our demonstration of the lower bounds in the finitary case introduces a novel method for encoding $\Sigma_N$ and $\Pi_N$ conditions into series inequalities.   Our demonstration of the lower bounds in the infinitary case is mostly straightforward.  However, our supporting result that computable infinitary logic can represent the inner product with $(2^{-(n+1)})_{n = 0}^\infty$ from the connectives $\neg$, $\intdiff$, $\frac{1}{2}$ appears to be new.  
Our examples in these demonstrations are somewhat artificial.  
We leave open directions such as the analysis of the theories of specific structures such as Lebesgue spaces or the construction of examples at different levels of complexity within natural classes such as Banach spaces or $C^*$-algebras.

\bibliographystyle{amsplain}
\bibliography{ourbib}

\end{document}